\documentclass{article}[12pt]
\usepackage{amsfonts,amsmath,amsthm,yfonts,mathdots,
latexsym,graphicx,amssymb,hyperref,url,fancyhdr,lineno,tikz,float}
\usepackage{amsfonts,amsmath,latexsym,hyperref,graphicx,amssymb,amsthm,url,lineno,float}
\usepackage[abbrev,msc-links]{amsrefs}  
\usetikzlibrary{arrows}
\usepackage{rotating}
\usetikzlibrary{calc}
\usetikzlibrary{positioning}
\usetikzlibrary{patterns}
\usetikzlibrary{shapes}
\usepackage{subcaption}

\title{Equally spaced collinear points in Euclidean Ramsey theory}

\author{Andrii Arman\thanks{University of Manitoba, \url{armana@myumanitoba.ca}}
\and Sergei Tsaturian\thanks{University of Manitoba, \url{tsaturis@cc.umanitoba.ca}}}
\date{11 May 2017}

\begin{document}
\maketitle

\begin{abstract}
It is proved that for $k\geq 4$, if the points of $k$-dimensional Euclidean space are coloured in red and blue, then there are either two red points distance one apart or $k+3$ blue collinear points with distance one between any two consecutive points. This result is new for $4\leq k\leq 10$.
\end{abstract}

\newtheorem{theorem}{Theorem}[section]
\newtheorem{lemma}[theorem]{Lemma}
\newtheorem{claim}[theorem]{Claim}
\newtheorem{corollary}[theorem]{Corollary}
\newtheorem{conjecture}{Conjecture}
\newtheorem{definition}[theorem]{Definition}
\newtheorem{problem} [theorem]{Problem}
\newtheorem{proposition} [theorem]{Proposition}
\newtheorem{example}[theorem]{Example}
\newtheorem{question}{Question}


\section{Introduction}
Let $\mathbb{E}^k$ be the $k$-dimensional Euclidean space and let $\ell_i$ denote the configuration of $i$ collinear points with distance $1$ between any two consecutive points. 
Say that two geometric configurations are congruent iff there exists an isometry (distance preserving bijection) between them. For $d\in\mathbb{Z}^+$, and geometric configurations $F_1$, $F_2$, let the notation $\mathbb{E}^d\rightarrow (F_1,F_2)$ mean that for any red-blue coloring of $\mathbb{E}^d$, either the red points contain a congruent copy of $F_1$, or the blue points contain a congruent copy of $F_2$.

It was asked by Erd\H os et al.~\cite{erdos2} if $\mathbb{E}^3\rightarrow (\ell_2,\ell_5)$ or even if $\mathbb{E}^2\rightarrow (\ell_2,\ell_5)$. The result of Iv\'an~\cite{ivan} implies the positive answer to the first question. Arman and Tsaturian~\cite{at} presented a simple proof of $\mathbb{E}^3\rightarrow (\ell_2,\ell_5)$ and proved a stronger result, namely that $\mathbb{E}^3\rightarrow (\ell_2,\ell_6)$. Tsaturian~\cite{tsaturian} proved that $\mathbb{E}^2\rightarrow (\ell_2,\ell_5)$. 

Denote by $m(k)$ the maximal number such that $\mathbb{E}^k\rightarrow (\ell_2, \ell_{m(k)})$, if it exists. Erd\H os and Graham~\cite{eg} claimed that $m(2)$ exists. The existence of $m(k)$ for all $k$ follows from a recent result by Conlon and Fox~\cite{cf}, who proved that $$(1+o(1))1.2^k<m(k)<10^{5k}.$$

In this short note, it is proved that for all $k \geq 4$, $m(k)\geq k+3$, which is better bound for small values of $k$, i.e. for $k\leq 10$. The techniques used here are not applicable when $k\leq 3$, so this note does not imply $\mathbb{E}^2\rightarrow (\ell_2,\ell_5)$ or $\mathbb{E}^3\rightarrow (\ell_2,\ell_6)$.

For a detailed overview of other results in Euclidean Ramsey theory, see Erd\H os et al.~\cite{erdos1}~\cite{erdos2}~\cite{erdos3} and Graham's survey \cite{graham}.

\section{Main result}
\begin{theorem}\label{thm:main}
For an integer $k \geq 4$, $\mathbb{E}^k\rightarrow (\ell_2, \ell_{k+3})$.
\end{theorem}
The following notation and preliminary lemmas are needed to prove Theorem~\ref{thm:main}. Denote by $\Delta^{k}$ any set of $k+1$ points in $\mathbb{E}^k$ such that distance between any two points in $\Delta^{k}$ is equal to one. In other words, $\Delta^{k}$ is a vertex set of a unit regular $k$-dimensional simplex in $\mathbb{E}^k$.
\begin{lemma}\label{lemma:s}
Let $k\geq 4$ and let the Euclidean space $\mathbb{E}^{k-1}$ be coloured in red and blue so that there are no two red points distance $1$ apart. Let $S^{k-2}$ be a $(k-2)$-dimensional sphere of radius $\frac{\sqrt{3}}{2}$ with the centre at point $O$. Then there is a copy of $\Delta^{k-2}\subset S^{k-2}$ all points of which are blue.
\end{lemma}
\begin{proof}
Assume the contrary, namely that there is no blue $\Delta^{k-2}\subset S^{k-2}$. Since all points in $\Delta^{k-2}$ are distance one to each other, it is equivalent to assume that any $\Delta^{k-2}\subset S^{k-2}$ contains exactly one red point. The following claim is the main part of the proof.\\

\noindent
\textsc{Claim.} There is an angle $\theta>0$, such that if $A$ is red point on $S^{k-2}$ and $B$ is antipodal point to $A$, then all points $C$ on $S^{k-2}$, such that $\angle COB=\theta$, are red. 

\noindent
\textsc{Proof of the claim.}
Let $A$ and $B$ be antipodal points on $S^{k-2}$ and let $A$ be red. Let $X$ be the set of points in $S^{k-2}$ that are at distance $1$ to $A$. Then $X$ is a $(k-3)$-dimensional sphere with radius $\sqrt{\frac{2}{3}}$. Let $A_1, A_2, \cdots A_{k-2}\in X$ be such that $\{A,A_1, A_2, \cdots A_{k-2}\}$ is a copy of $\Delta^{k-2}$. Since any simplex $\Delta^{k-2}$ contains exactly one red point and point $A$ is red, all points $A_1, A_2, \cdots A_{k-2}$ are blue. Let $A_{k-1}$ be the point symmetric to $A$ through the hyperplane $\pi$ spanned by points $A_1, A_2, \cdots A_{k-2}, O$. The point $A_{k-1}$ belongs to $S^{k-2}$ and is red, since $\{A_1, A_2, \cdots A_{k-1}\}$ is a copy of $\Delta^{k-2}$. Let $\theta= \angle A_{k-1}OB$, then $\theta>0$, because $\pi$ does not contain $X$. When the points $A_1, A_2, \cdots, A_{k-2}$ are rotated in $S^{k-2}$, the point $A_{k-1}$ spans the set of all points $C\in S^{k-2}$, such that $\angle COB=\theta$. This concludes the proof of the claim.\\

Let $A$ be a red point on $S^{k-2}$ and let $B$ be the antipodal point to $A$ on $S^{k-2}$.
Let $S^{k-3}_A\subset S^{k-2}$ be the set of all points $C$, such that $\angle COB=\theta$. By the Claim, all points of $S^{k-3}_A$ are red. For a point $C\in S^{k-3}_A$ let $C_1$ be the antipodal point on $S^{k-2}$. Let $S^{k-3}_C\subset S^{k-2}$ be the set of points $D$, such that $ \angle DOC_1=\theta$. By the Claim, the set $S^{k-3}_C$ contains only red points.
For a positive angle $\phi$, define a "hypercap" $HC_{A}(2\phi)=\{D\in S^{k-2} : \; \angle DOA \leq \phi\}$.
When $C$ is rotated in $S^{k-3}_A$, red hyper-circles $S^{k-3}_C$ span the red hypercap $HC_{A}(2\theta).$

The argument in last paragraph shows that if $A$ is a red point, then $HC_A(2\theta)$ is red. By reapplying this statement to any point in $HC_A(2\theta)$, it can be proved that the set $HC_{A}(4\theta)$ is red, the set $HC_{A}(8\theta)$ is red, and eventually the whole sphere $S^{k-2}$ is red. Hence, $S^{k-2}$ contains two red points distance 1 apart, which contradicts the assumption that $S^{k-2}$ does not contain a blue $\Delta^{k-2}$.
\end{proof}

For a positive integer $n$, denote by $[n]$ the set of all positive integers $i\leq n$.
\begin{lemma}\label{lemma:nored}
Let $\mathbb{E}^k$ be coloured in red and blue so that there is no red $\ell_2$. If there exists an integer $d$, $2\leq d\leq k+1$, and two red points distance $d$ apart, then there exists a blue $\ell_{k+3}$.
\end{lemma}
\begin{proof}

Let $A_0$ and $A_d$ be two red points distance $d$ apart. Assume that $A_0=(\frac{1}{2}, 0, \cdots, 0)$ and $A_d=(d+\frac{1}{2}, 0 \cdots, 0)$. 

For $0 \leq j \leq k+2$ define 
$$S^{k-2}_j=\{(j,x_2,\cdots,x_k): \; x_{2}^2+\cdots +x_k^2=\frac{3}{4}\}.$$ 
Note that $S^{k-2}_{0}$ and $S^{k-2}_{1}$ contain only blue points, since any point in $S^{k-2}_{0}$ or $S^{k-2}_{1}$ is distance one to $A_0$. For the same reason, sets $S^{k-2}_d$ and $S^{k-2}_{d+1}$ contain only blue points. Let $i\in[k+2]$ be a number not equal to $1,d$ or $d+1$.  By Lemma~\ref{lemma:s} applied to the hyperspace $x_{1}=i$ and $S^{k-2}_{i}$, there is a blue $\Delta^{k-2}\subset S^{k-2}_{i}$. Let $\Delta^{k-2}=\{A^{i}_{1}, A^{i}_{2}, \cdots A^{i}_{k-1} \}$. For all $0 \leq j \leq k+2$ and $s\in[k-1]$, define $$A^{j}_{s}=A^{i}_s+(j-i,0,0, \cdots, 0).$$
Let $C=[k+2]\backslash \{d, d+1, i\}$. For each $j \in C$, the set $\{A^{j}_1, \cdots A^{j}_{k-1}\}$ is a copy of $\Delta^{k-2}$, and therefore contains at most one red point. Since there are $k-2$ possible choices for $j\in C$ and there are $k-1$ possible choices for $s\in[k-1]$, there is an $s\in[k-1]$, such that for all $j\in C$, point $A^{j}_{s}$ is blue. Hence, points $A_{s}^{0}, A_{s}^{1}, \cdots , A_{s}^{k+2}$ are all blue and form a blue $\ell_{k+3}$.
\end{proof}
\noindent
{\bf Proof of Theorem~\ref{thm:main}.}
Assume the contrary, that there is a colouring of $\mathbb{E}^{k}$ in red and blue, such that there is neither red $\ell_2$, nor blue $\ell_{k+3}$.

According to Lemma~\ref{lemma:nored}, there are no two red points distance $1,2, \cdots, k+1$ apart. Let $A$ be a red point. Then for all $j\in [k+1]$, the sphere $$S^{k-1}(j)=\{X\in \mathbb{E}^k: |XA|=j\}$$ 
contains only blue points. Let $S^{k-1}(k+2)=\{X\in \mathbb{E}^k: |XA|=k+2\}$ and $S^{k-1}(k+3)=\{X\in \mathbb{E}^k: |XA|=k+3\}$. There are two cases to consider.

If $S^{k-1}(k+2)$ contains only blue points, let $P_{1}$ and $P_{2}$ be two points on $S^{k-1}(k+2)$, such that $|P_1P_2|=\frac{k+2}{k+3}$.
If $S^{k-1}(k+2)$ contains a red point $B$, let $P_{1}$ and $P_{2}$ be two points on $S^{k-1}(k+2)$, such that $|P_1P_2|=\frac{k+2}{k+3}$ and $|BP_1|=|BP_2|=1$.
In any case, both $P_1$ and $P_2$ are blue.

Let the lines $AP_1$ and $AP_2$ intersect hypersphere $S^{k-1}(k+3)$ at points $Q_1$ and $Q_2$ respectively. Then, $|Q_1Q_2|=1$, so one of the points, say, $Q_1$, is blue. For all $j\in[k+3]$ the line $AQ_1$ intersects the sphere  $S^{k-1}_{j}$ at a blue point, so the points of intersections form a blue $\ell_{k+3}$.
\qed

\section{Concluding remarks}
The result of Conlon and Fox~\cite{cf} (as well as the result of this note) implies that for any $k$, there is $n$ such that $\mathbb{E}^n\rightarrow (\ell_2,\ell_k)$. One of the results of Erd\H os et al.~\cite{erdos1} implies that for all $n$, $\mathbb{E}^n\not\rightarrow (\ell_6,\ell_6)$. This motivates the following question: what is the minimal $s$ such that there exists $k$ such that for all $n$, $\mathbb{E}^n\not\rightarrow (\ell_s,\ell_k)$? We conjecture that $s=3$:

\begin{conjecture}
There is an integer $k$, such that for every integer $n$ 
$$\mathbb{E}^n\not \rightarrow (\ell_3,\ell_k).$$
\end{conjecture}

During the preparation of this note, the paper of Conlon and Fox~\cite{cf} appeared, where the authors made a similar conjecture.

\section{Acknowledgments}
We would like to thank David Gunderson for valuable comments.

\end{document}